\documentclass[reqno,a4paper]{amsart}
\usepackage{amssymb}
\usepackage{amsmath}
\newcommand{\half}{\frac{1}{2}}

\begin{document} 
\newtheorem{prop}{Proposition}[section]
\newtheorem{Def}{Definition}[section] \newtheorem{theorem}{Theorem}[section]
\newtheorem{lemma}{Lemma}[section] \newtheorem{Cor}{Corollary}[section]

\title[ MKG in temporal gauge]{\bf Local well-posedness for the (n+1) - dimensional Maxwell-Klein-Gordon equations in temporal gauge}
\author[Hartmut Pecher]{
{\bf Hartmut Pecher}\\
Fakult\"at f\"ur Mathematik und Naturwissenschaften\\
Bergische Universit\"at Wuppertal\\
Gau{\ss}str.  20\\
42119 Wuppertal\\
Germany\\
e-mail {\tt pecher@math.uni-wuppertal.de}}
\date{}

\begin{abstract}
This is an extension of the paper [14] by the author for the 2+1 dimensional Maxwell-Klein-Gordon equations in temporal gauge to the n+1 dimensional situation for $n \ge 3$. They are shown to be locally well-posed for low regularity data, in 3+1 dimensions even below energy level improving a result by Yuan. Fundamental for the proof is a partial null structure of the nonlinearity which allows to rely on bilinear estimates in wave-Sobolev spaces, in 3+1 dimensions proven by d'Ancona, Foschi and Selberg, on an $(L^{\frac{2(n+1)}{n-1}}_x L^2_t)$ - estimate for the solution of the wave equation, and on the proof of a related result for the Yang-Mills equations by Tao.

\end{abstract}
\maketitle
\renewcommand{\thefootnote}{\fnsymbol{footnote}}
\footnotetext{\hspace{-1.5em}{\it 2010 Mathematics Subject Classification:} 
35Q40, 35L70 }
\normalsize 
\setcounter{section}{0}
\section{Introduction and main results}

\noindent Consider the Maxwell-Klein-Gordon equations
\begin{align}
\label{1.1}
\partial^{\alpha} F_{\alpha \beta} & = -Im(\phi \overline{D_{\beta} \phi}) \\
\label{1.2}
D^{\mu}D_{\mu} \phi &= m^2 \phi
\end{align}
in Minkowski space $\mathbb{R}^{1+n} = \mathbb{R}_t \times \mathbb{R}^n_x$ with metric $diag(-1, ...,1)$. Greek indices run over $\{0,1,...,n\}$, Latin indices over $\{1,...,n\}$, and the usual summation convention is used.  Here $m \in \mathbb{R}$ and
$$ \phi: \mathbb{R}\times \mathbb{R}^n \to \mathbb{C} \, , \, A_{\alpha}: \mathbb{R} \times \mathbb{R}^n \to \mathbb{R} \, , \, F_{\alpha \beta} = \partial_{\alpha} A_{\beta} - \partial_{\beta} A_{\alpha} \, , \, D_{\mu} = \partial_{\mu} + iA_{\mu} \, . $$
$A_{\mu}$ are the gauge potentials, $F_{\mu \nu}$ is the curvature. We use the notation $\partial_{\mu} = \frac{\partial}{\partial x_{\mu}}$, where we write $(x^0,x^1,...,x^n)=(t,x^1,...,x^n)$ and also $\partial_0 = \partial_t$.

Setting $\beta =0$ in (\ref{1.1}) we obtain the Gauss-law constraint
\begin{equation}
\label{1.3}
\partial^j F_{j 0} = -Im(\phi \overline{D_0 \phi})  \, .
\end{equation}
The system (\ref{1.1}),(\ref{1.2}) is invariant under the gauge transformations
$$A_{\mu} \to A'_{\mu} = A_{\mu} + \partial_{\mu} \chi \, , \, \phi \to \phi' = e^{i \chi} \phi \, , \, D_{\mu} \to D'_{\mu} = \partial_{\mu} + i A'_{\mu} \, .$$
This allows to impose an additional gauge condition. We exclusively consider the temporal gauge
\begin{equation}
\label{1.5}
A_0 = 0 \, .
\end{equation}
In this gauge the system (\ref{1.1}),(\ref{1.2}) is equivalent to
\begin{align}
\label{1.6}
\partial_t \partial^j A_j &=  Im(\phi \overline{\partial_t \phi}) \\
\label{1.6'}
\Box A_j &= \partial_j(\partial^k A_k) - Im(\phi \overline{\partial_j \phi}) +  A_j |\phi|^2 \\
\label{1.6''}
\Box \phi &= -i (\partial^k A_k) \phi - 2i A^k \partial_k \phi + A^k A_k \phi  + m^2 \phi\, ,
\end{align}
where $\Box = -\partial_t^2 + \Delta$ is the d'Alembert operator. 

Other choices of the gauge are the Coulomb gauge $\partial^j A_j =0$ and the Lorenz gauge $\partial^{\mu} A_{\mu} = 0$. 

The classical (3+1)-dimensional Maxwell-Klein-Gordon system has been studied by Klainerman and Machedon \cite{KM} where the existence of global solutions for data in energy space and above in Coulomb gauge was shown. Uniqueness in a suitable subspace was also shown. For the temporal gauge they also showed a similar result by using a suitable gauge transformation applied to the solution constructed in Coulomb gauge. They made use of a null structure for the main bilinear term to achieve this result.  Local well-posedness in Coulomb gauge for data for $\phi$ in the Sobolev space $H^s$ and for $A$ in $H^r$ with $r=s > 1/2$, i.e., almost down to the critical space with repect to scaling, was shown by Machedon and Sterbenz \cite{MS}. Global well-posedness below energy space (for $r=s > \sqrt{3}/2$) in Coulomb gauge was shown by Keel, Roy and Tao \cite{KRT}.

The problem in Lorenz gauge was considered by Selberg and Tesfahun \cite{ST}, who detected a null structure also in this case, and proved global well-posedness in energy space, especially also unconditional uniqueness in this space. The author \cite{P} proved local well-posedness for $s=\frac{3}{4}+\epsilon$ and $r=\frac{1}{2}+\epsilon$.

The problem in temporal gauge was treated by Yuan \cite{Y} directly in 
$X^{s,b}$-spaces. He stated local well-posedness in $X^{s,b}$-spaces for large data for $\phi$ in  $H^s$ and for $A$ in $H^r$ with $r=s > 3/4 $, where he just referred to the estimates given for Tao's small data local well-posedness results \cite{T1} in the Yang-Mills case. As a consequence  he
proved existence of a global solution in energy space and also uniqueness in subspaces of $X^{s,b}$-type. Unconditional uniqueness in the natural solution space in the finite energy case was shown by the author \cite{P1}. These results in temporal gauge rely on a similar result by Tao \cite{T1} for the Yang-Mills equations and small data. 

All these results were given in the (3+1)-dimensional case.

In 2+1 dimensions local well-posedness in Lorenz gauge for $s=\frac{3}{4}+\epsilon$ and $r=\frac{1}{4}+\epsilon$ was shown by the author \cite{P}. In Coulomb gauge local well-posedness for $s=r=\frac{1}{2}+\epsilon$ and also for $s=\frac{5}{8}+\epsilon$ , $r=\frac{1}{4}+\epsilon$ was obtained by Czubak and Pikula \cite{CP}, which was slightly improved to the case  $s=\frac{1}{2}+\epsilon$ , $r=\frac{1}{4}+\epsilon$ in \cite{P3}. In the temporal gauge in\cite{P3} local well-posedness was shown for data under the minimal smoothness assumption $s=r=\frac{1}{2}+\frac{1}{12}+\epsilon$ .
 
In the present paper we consider the (n+1)-dimensional case in the temporal gauge and lower down the minimal regularity assumptions on the data further using similar methods as in the (2+1)-dimensional case in \cite{P3}. We prove local well-posedness for data for $\phi$ in $H^s$ and $A$ in $H^r$, where $ s > \frac{n}{2} - \frac{3}{4}$ and $ r > \frac{n}{2} - 1$, where uniqueness holds in $X^{s,b}$ spaces (for more precise assumptions cf. Theorem \ref{Theorem'} ). The critical case with respect to scaling is $r=s=\frac{n}{2}-1$ , which we almost reach with respect to $r$. For technical reasons it is necessary to assume in a first step that the curl-free part of $A(0)$ vanishes (cf. Proposition \ref{Theorem}). This condition is removed by a suitable gauge transformation afterwards, which preserves  the regularity of the solution.
We need the null structure of some of the nonlinearities, the bilinear estimates for wave-Sobolev spaces $X^{s,b}_{|\tau|=|\xi|}$ , which were formulated by d'Ancona, Foschi and Selberg \cite{AFS} in arbitrary dimensions and proven in the case $n \le 3$, a generalization of a special casde to higher dimensions by \cite{P4}, and Tao's hybrid estimates \cite{T1} for the product of functions in wave-Sobolev spaces $X^{s,b}_{|\tau|=|\xi|}$ and in product Sobolev spaces $X^{l,b}_{\tau = 0}$ (cf. the definition of the spaces below) which have to be generalized from the special case $n=3$ and $l=s+\frac{1}{4}$. Moreover we need an appropriate generalization of the estimates for the terms which fulfill a null condition. Of fundamental importance is an $(L^{\frac{2(n+1)}{n-1}}_x L^2_t)$ - estimate for the solution of the wave equation which goes back to Tataru \cite{KMBT} and Tao \cite{T1}.
        
We denote both the Fourier transform with respect to space and time and with respect to space by $\,\widehat{\cdot}\,\,$ or ${\mathcal F}$. The operator
$D^{\alpha}$ is defined by $({\mathcal F}(D^{\alpha} f))(\xi) = |\xi|^{\alpha} ({\mathcal F}f)(\xi)$ and similarly $ \Lambda^{\alpha}$ by $({\mathcal F}(\Lambda^{\alpha} f))(\xi) = \langle \xi \rangle^{\alpha} ({\mathcal F}f)(\xi)$ , where we define $\langle \,\cdot\, \rangle := (1+|\cdot|^2)^{\frac{1}{2}}$ . The inhomogeneous Sobolev spaces are denoted by $H^{s,p}$. For $p=2$ we simply denote them by $H^s$. We repeatedly use the Sobolev embeddings $H^{s,p} \hookrightarrow L^q$ for $\frac{1}{p} \ge \frac{1}{q} \ge \frac{1}{p}-\frac{s}{n}$ and $1<p\le q < \infty$ . We also use the notation
$a \pm := a \pm \epsilon$ for a sufficiently small $\, \epsilon >0$ .

The standard space $X^{s,b}_{\pm}$ of Bourgain-Klainerman-Machedon type (which were already considered by M. Beals \cite{B}) belonging to the half waves is the completion of the Schwarz space  $\mathcal{S}({\mathbb R}^{n+1})$ with respect to the norm
$$ \|u\|_{X^{s,b}_{\pm}} = \| \langle \xi \rangle^s \langle  \tau \pm |\xi| \rangle^b \widehat{u}(\tau,\xi) \|_{L^2_{\tau \xi}} \, . $$ 
The wave-Sobolev space $H^{s,b}$  is the completion of the Schwarz space  $\mathcal{S}({\mathbb R}^{n+1})$ with respect to the norm
$$ \|u\|_{H^{s,b}} = \| \langle \xi \rangle^s \langle  |\tau| - |\xi| \rangle^b \widehat{u}(\tau,\xi) \|_{L^2_{\tau \xi}}  $$
and also $X^{s,b}_{\tau =0}$ with norm 
$$\|u\|_{X^{s,b}_{\tau=0}} = \| \langle \xi \rangle^s \langle  \tau  \rangle^b \widehat{u}(\tau,\xi) \|_{L^2_{\tau \xi}} \, .$$

We also define $X^{s,b}_{\pm}[0,T]$ as the space of the restrictions of functions in $X^{s,b}_{\pm}$ to $[0,T] \times \mathbb{R}^n$ and similarly $H^{s,b}[0,T]$ and $X^{s,b}_{\tau =0}[0,T]$ . We frequently use the estimate $\|u\|_{X^{s,b}_{\pm}} \le \|u\|_{H^{s,b}}$ for $b \le 0$ and the reverse estimate for $b \ge 0$. This allows to replace the spaces $X^{s,b}_{\pm}$ by $H^{s,b}$ in the nonlinear estimates.

We decompose $A=(A_1,...,A_n)$ into its divergence-free part $A^{df}$ and its curl-free part $A^{cf}$ :
\begin{equation}
\label{1.9}
A = A^{df} + A^{cf} \, ,
\end{equation}
where
\begin{equation}
\label{1.10}
 A_j^{df} = R^k(R_j A_k - R_k A_j) \quad , \quad A_j^{cf} = - R_j R_k A^j \, ,
 \end{equation}
and $R_k := D^{-1} \partial_k$ are the Riesz operators.
Let $PA := A^{df}$ denote the projection operator onto the divergence free part. Then we obtain the equivalent system
\begin{align}
\label{1.11}
\partial_t A^{cf} &= - D^{-2} \nabla Im(\phi \overline{\partial_t \phi}) \\
\label{1.12}
\Box A^{df} & = -P( Im(\phi \overline{ \nabla \phi}) + iA |\phi|^2) \\
\label{1.13}
\Box \phi & = i(\partial^j A_j^{cf}) \phi + 2i A^{df}_j \partial^j \phi +2i A_j^{cf} \partial^j \phi + A^j A_j \phi \, ,
\end{align}
where $A$ is replaced by (\ref{1.9}).

Klainerman and Machedon detected that $A^{df} \cdot \nabla \phi$ and $P(Im(\phi \overline{\nabla \phi})_k)$ are null forms. An elementary calculation namely shows that
\begin{align}
\label{1.13'}
 2A^{df}_i \partial^i \phi & = Q_{ij}(\phi,|\nabla|^{-1}(R^i A^j- R^j A^i)) \end{align}
 and
 \begin{align}
\label{1.12'}
P(Im(\phi \overline{\nabla \phi})_k) & = -2 R^j |\nabla|^{-1} Q_{kj}(Re \phi, Im \phi) 
\end{align}
where the null form $Q_{ij}$ is defined by
$$ Q_{ij}(u,v):= \partial_i u \partial_j v - \partial_j u \partial_i v \, .$$

Defining
\begin{align*}
\phi_{\pm} = \frac{1}{2}(\phi \pm i \Lambda^{-1} \partial_t \phi)&
 \Longleftrightarrow \phi=\phi_+ + \phi_- \, , \, \partial_t \phi = i \Lambda (\phi_+ - \phi_-) \\
 A^{df}_{\pm} = \frac{1}{2}(A^{df} \pm i \Lambda^{-1} \partial_t A^{df}) & \Longleftrightarrow A^{df} = A^{df}_+ + A_-^{df} \, , \, \partial_t A^{df} = i \Lambda(A^{df}_+ - A^{df}_-)
 \end{align*}
 we can rewrite (\ref{1.11}),(\ref{1.12}),(\ref{1.13}) as
 \begin{align}
 \label{1.11*}
 \partial_t A^{cf} &= - D^{-2} \nabla Im(\phi \overline{\partial_t \phi}) \\
 \label{1.12*}
(-i \partial_t \pm \Lambda)A_{j \pm} ^{df} & = \mp 2^{-1} \Lambda^{-1} ( 2R^j D^{-1} Q_{kj}(Re \phi, Im \phi) +iA_j |\phi|^2 - A_j^{df}) \\
\nonumber
(-i \partial_t \pm \Lambda) \phi_{\pm} &= \mp 2^{-1} \Lambda^{-1}( i(\partial^j A_j^{cf})\phi +  iQ_{kj}(\phi,|\nabla|^{-1}(R^k A^j- R^j A^k)) \\ \label{1.13*}
& \quad+2i A_j^{cf} \partial^j \phi + A^j A_j \phi- \phi) \, .
\end{align}
The initial data are transformed as follows:
\begin{align}
\label{1.14*}
\phi_{\pm}(0) &= \frac{1}{2}(\phi(0) \pm i^{-1} \Lambda^{-1} (\partial_t \phi)(0)) \\
\label{1.15*}
A^{df}_\pm(0) & = \frac{1}{2}(A^{df}(0) \pm i^{-1} \Lambda^{-1} (\partial_t A^{df})(0)) \, .
\end{align}

Our main result is preferably formulated in terms of the system (\ref{1.6}),(\ref{1.6'}),(\ref{1.6''}).

\begin{theorem}
\label{Theorem'}
Let $ n \ge 3$ . \\
1. Assume $\frac{n}{2} - \half \ge r> \frac{n}{2}-1$ , $s> \frac{n}{2} -\frac{3}{4}$ , $2r-s > \frac{n}{2}-\frac{3}{2} $ , $ r \ge s-1$ , $ l > \frac{n-1}{2} $, $l \le 1+s$ , $ l < 2s-\frac{n}{2}+1$ . Let $\phi_0 \in H^s({\mathbb R}^n),$  $\phi_1 \in H^{s-1}({\mathbb R}^n),$  
$a_0 \in H^r({\mathbb R}^n)$ , $a_1 \in H^{r-1}({\mathbb R}^n)\,$ be given,
which satisfy the compatibility condition
\begin{equation}
\label{CC}
\partial_j a^j_1 = Im(\phi_0 \overline{\phi}_1) 
\end{equation}
Then there exists $T>0$, such that (\ref{1.6}),(\ref{1.6'}),(\ref{1.6''}) with initial conditions
$ \phi(0)= \phi_0$ , $(\partial_t \phi)(0) = \phi_1$ , $A(0) = a_0$ , $(\partial_t A)(0)= a_1$ 
has a unique local solution
$$ \phi= \phi_++ \phi_- \quad , \quad A=A_+ + A_- + \tilde{A}$$
with 
$$ \phi_{\pm} \in X^{s,\frac{1}{2}+\epsilon}_{\pm}[0,T] \, , \, A_{\pm} \in X^{r,\frac{n}{2}-r+\epsilon}_{\pm}[0,T] \, , \,  \tilde{A} \in X^{l,\frac{1}{2}+\epsilon-}_{\tau =0}[0,T] \, ,$$
where $\epsilon >0$ is sufficiently small.\\
2. This solution satisfies
$$\phi \in C^0([0,T],H^s({\mathbb R}^n)) \cap C^1([0,T],H^{s-1}({\mathbb R}^n))\, , $$
$$ A \in C^0([0,T],H^r({\mathbb R}^n))\cap C^1([0,T],H^{r-1}({\mathbb R}^n)) \, . $$
\end{theorem}

\section{Basic tools}
Fundamental for us are the following estimates. We frequently use the classical Sobolev multiplication law in dimension $n$ :
\begin{equation}
\label{SML}
\|uv\|_{H^{-s_0}} \lesssim \|u\|_{H^{s_1}} \|v\|_{H^{s_2}} \, ,
\end{equation}
if $s_0 + s_1+s_2 \ge \frac{n}{2}$ and $s_0+s_1+s_2 \ge \max(s_0,s_1,s_2)$ , where at most one of these inequalities is an equality.

The corresponding  bilinear estimates in wave-Sobolev spaces were formulated in arbitrary dimension $n \ge 2$ and proven by d'Ancona, Foschi and Selberg in the case $n=3$ in \cite{AFS} and also proven in the case $n=2$ in \cite{AFS1} in a form which includes some more limit cases which we do not need.
\begin{prop}
\label{Prop.2}
For $s_0,s_1,s_2,b_0,b_1,b_2 \in {\mathbb R}$ and $u,v \in   {\mathcal S} ({\mathbb R}^{n+1})$ the estimate
$$\|uv\|_{H^{-s_0,-b_0}} \lesssim \|u\|_{H^{s_1,b_1}} \|v\|_{H^{s_2,b_2}} $$ 
holds, provided the following conditions are satisfied:
\begin{align*}
\nonumber
& b_0 + b_1 + b_2 > \frac{1}{2} \, ,
& b_0 + b_1 \ge 0 \, ,\quad \qquad  
& b_0 + b_2 \ge 0 \, ,
& b_1 + b_2 \ge 0
\end{align*}
\begin{align*}
\nonumber
&s_0+s_1+s_2 > \frac{n+1}{2} -(b_0+b_1+b_2) \\
\nonumber
&s_0+s_1+s_2 > \frac{n}{2} -\min(b_0+b_1,b_0+b_2,b_1+b_2) \\
\nonumber
&s_0+s_1+s_2 > \frac{n-1}{2} - \min(b_0,b_1,b_2) \\
\nonumber
&s_0+s_1+s_2 > \frac{n+1}{4} \\
 &(s_0 + b_0) +2s_1 + 2s_2 > \frac{n}{2} \\
\nonumber
&2s_0+(s_1+b_1)+2s_2 > \frac{n}{2} \\
\nonumber
&2s_0+2s_1+(s_2+b_2) > \frac{n}{2}
\end{align*}
\begin{align*}
\nonumber
&s_1 + s_2 \ge \max(0,-b_0) \, ,\quad
\nonumber
s_0 + s_2 \ge \max(0,-b_1) \, ,\quad
\nonumber
s_0 + s_1 \ge \max(0,-b_2)   \, .
\end{align*}
\end{prop}

The proof of the following special case in higher dimensions and its Corollary was given in \cite{P4}, Prop. 3.6 and Cor. 3.1.

\begin{prop}
\label{Prop.1.2}
Assume $n \ge 4$ and
\begin{align*}
s_0+s_1+s_2 > \frac{n-1}{2} \quad , \quad (s_0+s_1+s_2)+s_1+s_2 > \frac{n}{2} \, ,\\
 s_0+s_1 \ge 0 \quad , \quad s_0+s_2 \ge 0 \quad , \quad s_1+s_2 \ge 0 \, . 
\end{align*}
The following estimate holds:
$$ \|uv\|_{H^{-s_0,0}} \lesssim \|u\|_{H^{s_1,\half+}} \|v\|_{H^{s_2,\half+}} \, . $$
\end{prop}
\begin{Cor}
\label{Cor.1.1}
Under the assumptions of Prop. \ref{Prop.1.2}
$$ \|uv\|_{H^{-s_0,0}} \lesssim \|u\|_{H^{s_1,\half-}} \|v\|_{H^{s_2,\half-}}  \, . $$
\end{Cor}

Moreover we need the standard Strichartz type estimates for the wave equation  given in the next proposition.
\begin{prop}
\label{Str'}
If $n \ge 2$ and
$$
2 \le q \le \infty, \quad
2 \le r < \infty, \quad
\frac{2}{q} \le (n-1) \left(\half - \frac{1}{r} \right), 
$$
then the following estimate holds:
$$ \|u\|_{L_t^q L_x^r} \lesssim \|u\|_{H^{\frac{n}{2}-\frac{n}{r}-\frac{1}{q},\half+}} \, ,$$
especially
\begin{equation}
\label{Str}
\|u\|_{L^{\frac{2(n+1)}{n-1}}_{xt}} \lesssim \|u\|_{H^{\frac{1}{2},\frac{1}{2}+}} 
\end{equation}

\end{prop}
\begin{proof}
This the Strichartz type estimate, which can be found for e.g. in \cite{GV}, Prop. 2.1, combined with the transfer principle.
\end{proof}

Essential for us is the following estimate, which essentially goes back to Tataru \cite{KMBT} and Tao \cite{T}.
\begin{prop}
\label{Lemma}
The following estimates hold:
\begin{align*}
\|u\|_{L^{\frac{2(n+1)}{n-1}}_x L^2_t} & \lesssim \|u\|_{H^{\frac{n-1}{2(n+1)},\frac{1}{2}+}} \, , \\
\|u\|_{L^{\frac{2(n+1)}{n-1}}_x L^{2+}_t} & \lesssim \|u\|_{H^{\frac{n-1}{2(n+1)}+,\frac{1}{2}+}}  \, .
\end{align*}
\end{prop}
\begin{proof}
In the case $n=3$ one may simply refer to \cite{T}, Prop. 4.1. Alternatively
by \cite{KMBT}, Thm. B2 we obtain
$ \|{\mathcal F}_t u \|_{L^2_{\tau} L^{\frac{2(n+1)}{n-1}}_x} \lesssim \|u_0\|_{\dot{H}^{\frac{n-1}{2(n+1)} }} \, , $
if $u= e^{itD} u_0$ and ${\mathcal F}_t$ denotes the Fourier transform with respect to time. This implies by Plancherel and Minkowski's inequality
$$ \|u\|_{L^{\frac{2(n+1)}{n-1}}_x L^2_t} = \|{\mathcal F}_t u \|_{L^{\frac{2(n+1)}{n-1}}_x L^2_{\tau}} \le \|{\mathcal F}_t u \|_{L^2_{\tau} L^{\frac{2(n+1)}{n-1}}_x} \lesssim \|u_0\|_{\dot{H}^{\frac{n-1}{2(n+1)}}} \, . $$
The transfer principle \cite{S1}, Prop. 8 implies 
\begin{equation}
\|u\|_{L^{\frac{2(n+1)}{n-1}}_x L^2_t} \lesssim \|u\|_{H^{\frac{n-1}{2(n+1),\frac{1}{2}+}}} \, .
\end{equation}
Interpolation with (\ref{Str}) gives
\begin{equation}
\|u\|_{L^{\frac{2(n+1)}{n-1}}_x L^{2+}_t}  \lesssim \|u\|_{H^{\frac{n-1}{2(n+1)}+,\frac{1}{2}+}}  \, .
\end{equation}
\end{proof}

In order to estimate the null forms we  use the following estimate.
\begin{lemma}
\begin{align}
\label{Qij}
Q_{ij}(\phi,\psi)
& \precsim
D^{\half} D_-^{\half} (D^{\half} \phi  D^{\half} \psi) + D^{\half}(D^{\half} D_-^{\half} \phi D^{\half} \psi)  + D^{\half}(D^{\half} \phi D^{\half} D_-^{\half} \psi) \\
\nonumber
Q_{ij}(\phi,\psi) &\precsim
D^{\half-2\epsilon} D_-^{\half-2\epsilon} (D^{\half+2\epsilon} \phi  D^{\half+2\epsilon} \psi) + D^{\half-2\epsilon}(D^{\half+2\epsilon} D_-^{\half-2\epsilon} \phi D^{\half+2\epsilon} \psi) \\
\label{Qij1}
& \quad + D^{\half-2\epsilon}(D^{\half+2\epsilon} \phi D^{\half+2\epsilon} D_-^{\half-2\epsilon} \psi) \\ \nonumber
Q_{ij}(\phi,\psi)
& \precsim
D^{1-2\epsilon} D_-^{1-2\epsilon} (D^{2\epsilon} \phi  D^{2\epsilon} \psi) + D^{\half}(D^{\half} D_-^{\half} \phi D^{\half} \psi)  \\
\label{Qij2}
& \quad + D^{\half}(D^{\half} \phi D^{\half} D_-^{\half} \psi)
\end{align}
for $0 \le \epsilon \le \frac{1}{4}$  .
\end{lemma}
\begin{proof}
(\ref{Qij}) is proven in \cite{KM2}, whereas (\ref{Qij1}) and (\ref{Qij2}) follow by interpolation with the trivial estimate $Q_{ij}(\phi,\psi) \precsim D\phi D\psi$ .
\end{proof}

\section{Proof of Theorem \ref{Theorem'}}
For the proof it is essential to show that we may assume in a first step that the initial data satisfy $a_0^{cf} =0$ and that is is possible to cancel this condition in a second step by using a suitable gauge transformation.

\begin{prop}
\label{Theorem}
Let $ n \ge 3$ . \\
1. Assume $\frac{n}{2} - \half \ge r> \frac{n}{2}-1$ , $s> \frac{n}{2} -\frac{3}{4}$ , $2r-s > \frac{n}{2}-\frac{3}{2} $ , $ r \ge s-1$ , $ l > \frac{n-1}{2} $	, $l \le 1+s$ , $ l < 2s-\frac{n}{2}+1$ . Let $\phi_0 \in H^s({\mathbb R}^n),$  $\phi_1 \in H^{s-1}({\mathbb R}^n),$  
$a_0 \in H^r({\mathbb R}^n)$ , $a_1 \in H^{r-1}({\mathbb R}^n)\,$ be given,
which satisfy the compatibility condition
$$
\partial_j a^j_1 = Im(\phi_0 \overline{\phi}_1) 
$$
and
\begin{equation}
\label{cf}
a^{cf}_0 = 0 \, .
\end{equation}
Then there exists $T>0$, such that (\ref{1.11}),(\ref{1.12}),(\ref{1.13}) with initial conditions
$ \phi(0)= \phi_0,$  $(\partial_t \phi)(0) = \phi_1$ , $A(0) = a_0$ , $(\partial_t A)(0)= a_1$ 
has a unique local solution
$$ \phi= \phi_++ \phi_- \quad , \quad A=A^{df}_+ + A^{df}_- + A^{cf}$$
with 
$$ \phi_{\pm} \in X^{s,\frac{1}{2}+\epsilon}_{\pm}[0,T] \, , \, A^{df}_{\pm} \in X^{r,\frac{n}{2}-r+\epsilon}_{\pm}[0,T] \, , \,  A^{cf} \in X^{l,\frac{1}{2}+\epsilon-}_{\tau =0}[0,T] \, ,$$
where $\epsilon >0$ is sufficiently small.\\
2. This solution satisfies
$$\phi_{\pm} \in C^0([0,T],H^s({\mathbb R}^n)) \, , \, A^{df}_{\pm} \in C^0([0,T],H^r({\mathbb R}^n)) \, , $$
$$ A^{cf} \in C^0([0,T],H^{l}({\mathbb R}^n)) \cap C^1([0,T],H^{l-1}({\mathbb R}^n))\, . $$
\end{prop}

\begin{proof}[{\bf Proof of Proposition \ref{Theorem}}]
{\bf Proof of part 2:}
We assume for the moment that part 1 is true.
The compatability conditon (\ref{CC}), which is necessary in view of (\ref{1.3}), determines $a_1^{cf}$ as $a_1^{cf} = -(-\Delta)^{-1} \nabla(Im(\phi_0 \overline{\phi}_1))$ . \\
Is is not difficult to see that $a_1^{cf}$ fulfills $ a_1^{cf} \in H^{l-1}({\mathbb R}^n)$. One only has to show that
$$ \|D^{-1}(\phi_0 \overline{\phi}_1) \|_{H^{l-1}} \lesssim \|\phi_0\|_{H^s} \|\phi_1\|_{H^{s-1}} \ . $$
By duality this is equivalent to
$$\|\phi_0 \phi_2 \|_{H^{1-s}} \lesssim \|\phi_0\|_{H^s} \| D \phi_2\|_{H^{1-l}} \, . $$
In the case of high frequencies of $\phi_2$ this follows from the Sobolev multiplication law (\ref{SML}) using $2s-l > \frac{n}{2} -1$ , and the low frequency case can be easily handled using $s>\frac{1}{2}$. In the same way we also obtain from (\ref{1.11*}):
$\partial_t A^{cf} \in C^0([0,T],H^{l-1}({\mathbb R}^n)) \, .$ \\[0.5em]
{\bf Proof of part 1:} 
By a contraction argument the local existence and uniqueness proof is reduced to suitable multilinear estimates for the right hand sides of (\ref{1.11*}),(\ref{1.12*}),(\ref{1.13*}). For (\ref{1.12*}), e.g. , we make use of the following well-known estimate for a solution of the linear equation
$
(-i \partial_t \pm \Lambda)A_{\pm} ^{df}  = G$ , namely
$$
\|A^{df}_{\pm}\|_{X^{k,b}_{\pm}[0,T]} \lesssim \|A^{df}_{\pm}(0)\|_{H^k} + T^{b'-b} \| G \|_{X^{k,b'-1}_{\pm}[0,T]} \, , $$
which holds for $k\in{\mathbb R}$  , $\frac{1}{2} < b \le b' < 1$ and $0<T \le 1$ .

Thus the local existence and uniqueness for large data (in which case we have to choose $b < b'$) , in the regularity class $$ \phi_{\pm} \in X^{s,\frac{1}{2}+\epsilon}_{\pm}[0,T] \, , \, A^{df}_{\pm} \in X^{r,\frac{n}{2}-r+\epsilon}_{\pm}[0,T] \, , \,  A^{cf} \in X^{l,\frac{1}{2}+\epsilon-}_{\tau =0}[0,T] $$  can be reduced to 
the following estimates, if we take the assumption $a^{cf}= 0$ into account (remark that we do not want to assume $a^{cf} \in H^l$ later):
\begin{align}
\label{27}
\| D^{-1} (\phi_1 \partial_t \phi_2)\|_{X^{l,-\frac{1}{2}+\epsilon}_{\tau=0}} &\lesssim \|\phi_1\|_{H^{s,\frac{1}{2}+\epsilon}} \|\phi_2\|_{H^{s,\frac{1}{2}+\epsilon}} \, ,
\\
\label{28}
\|D^{-1}Q_{ij}(\phi_1,\phi_2)\|_{H^{r-1,\frac{n}{2}-r-1+2\epsilon}} 
&\lesssim \|\phi_1\|_{H^{s,\frac{1}{2}+\epsilon}} 
\|\phi_2\|_{H^{s,\frac{1}{2}+\epsilon}} \, , \\
\label{28'}
\|Q_{ij}(D^{-1}\phi_1,\phi_2)\|_{H^{s-1,-\frac{1}{2}+2\epsilon}} & \lesssim \|\phi_1\|_{H^{r,\frac{n}{2}-r+\epsilon}} \|\phi_2\|_{H^{s,\frac{1}{2}+\epsilon}} \, , \\
\label{29}
\| \nabla A \phi \|_{H^{s-1,-\frac{1}{2}+2\epsilon}} +
\| A \nabla \phi \|_{H^{s-1,-\frac{1}{2}+2\epsilon}} &\lesssim \| A\|_{X^{l,\frac{1}{2}+\epsilon-}_{\tau =0}}  \|\phi\|_{H^{s,\frac{1}{2}+\epsilon}} \, ,
\end{align}
\begin{align}
\label{30}
\| A \phi_1 \phi_2 \|_{H^{r-1,\frac{n-1}{2}-r+2\epsilon}} &\lesssim  \min(\|A\|_{H^{r,\frac{n}{2}-r+\epsilon}}, \|A\|_{X^{l,\frac{1}{2}+\epsilon-}_{\tau =0}} ) \prod_{i=1}^2 \|\phi_i\|_{H^{s,\frac{1}{2}+\epsilon}} \, ,
\\
\label{30'}
\| A_1 A_2 \phi \|_{H^{s-1,-\frac{1}{2}+2\epsilon}} &\lesssim \prod_{i=1}^2 \min(\|A_i\|_{H^{r,\frac{n}{2}-r+\epsilon}},\| A_i\|_{X^{l,\frac{1}{2}+\epsilon-}_{\tau =0}} ) \|\phi\|_{H^{s,\frac{1}{2}+\epsilon}} \, .
\end{align}
{\bf Proof of (\ref{28'}):} We use (\ref{Qij1}) and reduce the claim to the following estimates:
\begin{align}
\label{70}
\|uv\|_{H^{s-\half-2\epsilon,0}} & \lesssim \|u\|_{H^{s-\half-2\epsilon,\half+\epsilon}} \|v\|_{H^{r+\half-2\epsilon,b}} \\
\label{71}
\|uv\|_{H^{s-\half-2\epsilon,-\half+2\epsilon}} & \lesssim \|u\|_{H^{s-\half-2\epsilon,3\epsilon}} \|v\|_{H^{r+\half-2\epsilon,b}} \\
\label{72}
\|uv\|_{H^{s-\half-2\epsilon,-\half+2\epsilon}} & \lesssim \|u\|_{H^{s-\half-2\epsilon,\half+\epsilon}} \|v\|_{H^{r+\half-2\epsilon,b-\half+2\epsilon}} \, ,
\end{align}
where $b = \frac{n}{2}-r+\epsilon$ .

The estimate (\ref{70}) follows from Prop. \ref{Prop.1.2} and Prop. \ref{Prop.2}, where we remark that $b >\half$, because $r < \frac{n}{2} - \half$ by assumption. Here the parameters are given by $s_0 = \half - s +2\epsilon$ , $b_0 =0$ , $s_1= s-\half-2\epsilon$ , $b_1=\half+\epsilon$ , $s_2 =r+\half-2\epsilon$ , $b_2= b$ , so that $s_0+s_1+s_2 = r+\half-2\epsilon > \frac{n-1}{2}$ and $s_0+s_1+s_2+s_1+s_2 > \frac{n}{2}$ under our assumption $ r > \frac{n}{2}-1$ .

 (\ref{71}) is by duality equivalent to  
$$\|vw\|_{H^{\half-s+2\epsilon,0}}  \lesssim \|v\|_{H^{r+\half-2\epsilon,b}} \|w\|_{H^{\half-s+2\epsilon,\half-2\epsilon}} \, . $$ 
We use Prop. \ref{Prop.2} and Cor. \ref{Cor.1.1} with parameters $s_0 = \half-s+2\epsilon$ , $b_0 =0$ , $s_1= r+\half-2\epsilon$ , $b_1=b$ , $s_2= \half-s+2\epsilon$ , $b_2=\half-2\epsilon$ , so that $s_0+s_1+s_2 = r + \half-2\epsilon > \frac{n-1}{2}$ and $s_0+s_1+s_2+s_1+s_2 = 2r-s-2\epsilon +\frac{3}{2} > \frac{n}{2}$ under our assumption $2r-s > \frac{n}{2} - \frac{3}{2}$.

The estimate (\ref{72}) is equivalent to
\begin{align}
\nonumber
&\int_* \frac{\widehat{u_1}(\xi_1,\tau_1)}{\langle \xi_1 \rangle^{s-\half-2\epsilon} \langle  |\tau_1| - |\xi_1| \rangle^{\frac{1}{2}+\epsilon}} \frac{\widehat{u_2}(\xi_2,\tau_2) }{\langle \xi_2 \rangle^{r+\half-2\epsilon} \langle |\tau_2| - |\xi_2| \rangle^{b-\frac{1}{2}+2\epsilon}} 
\frac{\widehat{u_3}(\xi_3,\tau_3) \langle \xi_3 \rangle^{s-\half-2\epsilon}}{\langle |\tau_3| - |\xi_3| \rangle^{\frac{1}{2}-2\epsilon}} \cdot\\
\label{34}
& \hspace{20em}  \cdot\, d\xi d\tau \lesssim \prod_{i=1}^3 \|u_i\|_{L^2_{xt}} \, .
\end{align}
The Fourier transforms are nonnegative without loss of generality.
Here * denotes integration over $\sum_{i=1}^3 \xi_i =0$ , $\sum_{i=1}^3 \tau_i =0$ and $d\xi d\tau = d\xi_1 d\xi_2 d\xi_3 d\tau$ . \\
Case 1: $|\xi_3| \le |\xi_1|$ . The left hand side of (\ref{34}) is estimated by
\begin{align*}
&\int_* \frac{\widehat{u_1}(\xi_1,\tau_1)}{ \langle  |\tau_1| - |\xi_1| \rangle^{\frac{1}{2}+\epsilon}} \frac{\widehat{u_2}(\xi_2,\tau_2) }{\langle \xi_2 \rangle^{r+\half-2\epsilon} \langle |\tau_2| - |\xi_2| \rangle^{b-\frac{1}{2}+2\epsilon}} 
\frac{\widehat{u_3}(\xi_3,\tau_3) }{  \langle    |\tau_3| - |\xi_3| \rangle^{\frac{1}{2}-2\epsilon}}   \, d\xi d\tau  \\
&\lesssim \|{\mathcal F}^{-1}(\frac{\widehat{u_1}}{\langle  |\tau_1| - |\xi_1| \rangle^{\frac{1}{2}+\epsilon}})\|_{L^4_t L^2_x} \| {\mathcal F}^{-1}(\frac{\widehat{u_2}}{\langle \xi_2 \rangle^{r+\half-2\epsilon} \langle |\tau_2| - |\xi_2| \rangle^{b-\frac{1}{2}+2\epsilon}})\|_{L^2_t L^{\infty}_x} \cdot\\
&\hspace{10em}\cdot\|{\mathcal F}^{-1}(\frac{\widehat{u_3}}{\langle  |\tau_3| - |\xi_3| \rangle^{\frac{1}{2}-2\epsilon}})\|_{L^4_t L^2_x}  
\lesssim \prod_{i=1}^3 \|u_i\|_{L^2_{xt}} \, .
\end{align*}
For the second factor we interpolate the Strichartz estimate
$$ \|u\|_{L^2_t L^{\infty -}_x} \lesssim \|u\|_{H^{\frac{n-1}{2}+,\half +}} $$
and the Sobolev estimate
$$ \|u\|_{L^2_t L^{\infty -}_x} \lesssim \|u\|_{H^{\frac{n}{2},0}} \, , $$
which gives
$$ \|u\|_{L^2_t L^{\infty -}_x} \lesssim \|u\|_{H^{r+\half-3\epsilon+,\frac{n-1}{2}-r+3\epsilon }} $$
using our assumption $\frac{n}{2}-1 < r < \frac{n-1}{2}$ . This implies immediately by Sobolev
$$ \|u\|_{L^2_t L^{\infty}_x} \lesssim \|u\|_{L^2_t H^{0+,\infty -}_x} \lesssim \|u\|_{H^{r+\half-2\epsilon,b-\half+2\epsilon }} $$
by our choice $b=\frac{n}{2}-r+\epsilon$ , as desired. \\
Case 2: $|\xi_3| \gg |\xi_1|$ , thus $|\xi_2| \sim |\xi_3| \gg |\xi_1|$ . The left hand side of (\ref{34}) is estimated by
\begin{align*}
&\int_* \frac{\widehat{u_1}(\xi_1,\tau_1)}{ \langle \xi_1 \rangle^{s-\half-2\epsilon}\langle  |\tau_1| - |\xi_1| \rangle^{\frac{1}{2}+\epsilon}} \frac{\widehat{u_2}(\xi_2,\tau_2) }{\langle \xi_2 \rangle^{r-s+1} \langle |\tau_2| - |\xi_2| \rangle^{b-\frac{1}{2}+2\epsilon}} 
\frac{\widehat{u_3}(\xi_3,\tau_3) }{\langle |\tau_3| - |\xi_3| \rangle^{\frac{1}{2}-2\epsilon}}   \, d\xi d\tau  \\
&\lesssim\int_* \frac{\widehat{u_1}(\xi_1,\tau_1)}{ \langle \xi_1 \rangle^{r+\half-2\epsilon}\langle  |\tau_1| - |\xi_1| \rangle^{\frac{1}{2}+\epsilon}} \frac{\widehat{u_2}(\xi_2,\tau_2) }{ \langle |\tau_2| - |\xi_2| \rangle^{b-\frac{1}{2}+2\epsilon}} 
\frac{\widehat{u_3}(\xi_3,\tau_3) }{\langle|\tau_3| - |\xi_3| \rangle^{\frac{1}{2}-2\epsilon}}   \, d\xi d\tau  \\
&\lesssim \|{\mathcal F}^{-1}(\frac{\widehat{u_1}}{\langle \xi_1 \rangle^{r+\half-2\epsilon} \langle  |\tau_1| - |\xi_1| \rangle^{\frac{1}{2}+\epsilon}})\|_{L^{2}_t L^{\infty}_x} \| {\mathcal F}^{-1}(\frac{\widehat{u_2}}{\langle |\tau_2| - |\xi_2| \rangle^{\frac{n-1}{2}-r+3\epsilon}})\|_{L^{2+}_t L^2_x} \cdot\\
&\hspace{10em}\cdot\|{\mathcal F}^{-1}(\frac{\widehat{u_3}}{\langle  |\tau_3| - |\xi_3| \rangle^{\frac{1}{2}-2\epsilon}})\|_{L^{\infty -}_t L^2_x}  
\lesssim \prod_{i=1}^3 \|u_i\|_{L^2_{xt}} \, , 
\end{align*}
where we used Strichartz estimate and Sobolev for the first factor similarly as in Case 1 under the condition that $\frac{n}{2}-1  < r  < \frac{n-1}{2}$, and recalling our choice of $b$. This completes the proof of (\ref{28'}). \\
{\bf Proof of (\ref{28}):} We control $Q_{ij}(u,v)$ by (\ref{Qij2}) which reduces (\ref{28}) by symmetry to the following estimates:
\begin{align}
\label{73}
\|uv\|_{H^{r-1-b-,0}} & \lesssim \|u\|_{H^{s-b-,\half+\epsilon}} \|v\|_{H^{s-b-,\half+\epsilon}} \, ,\\
\label{74}
\|uv\|_{H^{r-\frac{3}{2},b-1+}} & \lesssim \|u\|_{H^{s-\half,0}} \|v\|_{H^{s-\half,\half+\epsilon}} \, ,
\end{align}
where $b= \frac{n}{2}-r+\epsilon$ . If we use (\ref{Qij1}) instead  we may replace (\ref{73}) by
\begin{equation}
\label{75}
\|uv\|_{H^{r-\frac{3}{2},b-\half+}} \lesssim \|u\|_{H^{s-\half,\half+\epsilon}} \|v\|_{H^{s-\half,\half+\epsilon}} \, .
\end{equation}

If $n=3$ we prove (\ref{75}) by use of Prop. \ref{Prop.2} (which holds in this case) with parameters
$s_0= \frac{3}{2}-r$ , $b_0 = \half-b-$ , $s_1=s_2=s-\half$ , $b_1=b_2= \half+\epsilon$ . The conditions of Prop. \ref{Prop.2} are satisfied, because $s_0+s_1+s_2 = 2s-r+\half > n-1-r+\epsilon + = \frac{n-1}{2} - b_0$ for $s> \frac{n}{2}-\frac{3}{4}$ . Moreover $s_0+s_1+s_2 > \frac{n}{2} - \half \ge \frac{n+1}{4}$ for $n \ge 3$ and $r < \frac{n}{2} - \half$ , and also $s_0+s_1+s_2+s_1+s_2+b_0 > n-1-r + 2s-1+\half-\frac{n}{2}+r > \frac{3}{2}n-3 \ge \frac{n}{2}$ for $s > \frac{n}{2}-\frac{3}{4}$ and $n \ge 3$ . Furthermore $s_0+s_1+s_2+s_0+s_2+b_1 = 3s-2r+2+\epsilon > \frac{n}{2}$ under our assumptions $s > \frac{n}{2}-\frac{3}{4}$ and $r < \frac{n-1}{2}$ , and finally $s_1+s_2 > - b_0$ $\Leftrightarrow$ $2s+r > \frac{n+1}{2}$ , which holds for $n \ge 3$ .

If $n \ge 4$ we now prove (\ref{73}) by use of Prop. \ref{Prop.1.2} with parameters $s_0 = b+1-r+$, $s_1=s_2= s-b-$ , so that $s_0+s_1+s_2 = 2s-r+1-b- = 2s+1-\frac{n}{2}-\epsilon- > \frac{n-1}{2}$ under our assumption $s > \frac{n}{2}-\frac{3}{4}$ . Moreover $s_1+s_2 = 2s-2b- = 2s-n+2r-2\epsilon- > n - \frac{7}{2} \ge \half$ under our assumptions on $s$ and $r$ , if $n \ge 4$ , so that $s_0+s_1+s_2+s_1+s_2 > \frac{n}{2}$ .

It remains to prove (\ref{74}). First we consider the case $r=\frac{n}{2}-1+\epsilon+$ , so that $b-1+ =0 $ and $\frac{3}{2}-r+2s-1 > \frac{n}{2}$ for $s > \frac{n}{2}-\frac{3}{4}$ . The estimate follows immediately by the Sobolev multiplication law (\ref{SML}). 
Next let $r= \frac{n-1}{2}$ , so that we have to prove
$$\|uv\|_{H^{\frac{n}{2}-2,-\half+\epsilon+}} \lesssim \|u\|_{H^{s-\half,0}} \|v\|_{H^{s-\half,\half+\epsilon}} \, .$$
This follows from Cor. \ref{Cor.1.1} in the case $n \ge 4$ and Prop. \ref{Prop.2} in the case $n=3$ with parameters $s_0=s-\half$ , $s_1=s-\half$ , $s_2= 2-\frac{n}{2}$ , so that $s_0+s_1+s_2=2s+1-\frac{n}{2} > \frac{n-1}{2}$ and $s_1+s_2 > \frac{3}{4}$ . 
The general case $ \frac{n}{2}-1+\epsilon+ < r <\frac{n-1}{2}$ follow by interpolation of these two cases, as one easily checks.

{\bf Proof of (\ref{27}):} 
We first remark that the singularity of $D^{-1}$ is harmless in $ n \ge 3$ dimensions (\cite{T}, Cor. 8.2) and it can be replaced by $\Lambda^{-1}$. As a first step we use Sobolev's multiplication law (\ref{SML}) and obtain
\begin{align*}
\big|\int \int u_1 u_2 u_3 dx dt\big| \lesssim \|u_1\|_{X^{s,\frac{1}{2}+\epsilon}_{\tau =0}}
\|u_2\|_{X^{s,-\frac{1}{2}+\epsilon}_{\tau =0}}
\|u_3\|_{X^{1-l,\frac{1}{2}-\epsilon}_{\tau =0}} 
\end{align*}
provided that $l < 2s-\frac{n}{2}+1$ and $ s \ge l-1  $ ,  which is fulfilled under our assumptions. This implies taking the time derivative into account 
\begin{align}
\label{40}
&\| \Lambda^{-1} (\phi_1 \partial_t \phi_2)\|_{X^{l,-\frac{1}{2}+\epsilon}_{\tau=0}} \lesssim \|\phi_1\|_{X^{s,\frac{1}{2}+\epsilon}_{\tau=0}} \|\phi_2\|_{X^{s,\frac{1}{2}+\epsilon}_{\tau=0}} \, .
\end{align}
In a second step we want to prove
\begin{align}
\label{41}
&\| \Lambda^{-1} (\phi_1 \partial_t \phi_2)\|_{X^{l,-\frac{1}{2}+\epsilon}_{\tau=0}} + \| \Lambda^{-1} (\phi_2 \partial_t \phi_1)\|_{X^{l,-\frac{1}{2}+\epsilon}_{\tau=0}} 
\lesssim \|\phi_1\|_{X^{s,\frac{1}{2}+\epsilon}_{|\tau|=|\xi|}} \|\phi_2\|_{X^{s,\frac{1}{2}+\epsilon}_{\tau=0}} \, .
\end{align}
If $\widehat{\phi_1}(\xi_3,\tau_3)$ is supported in  $ ||\tau_3|-|\xi_3|| \gtrsim |\xi_3| $ we have the trivial bound
\begin{equation}
\label{41'}
 \|\phi_1\|_{X^{s,\frac{1}{2}+\epsilon}_{\tau=0}} \lesssim \|\phi_1\|_{X^{s,\frac{1}{2}+\epsilon}_{|\tau|=|\xi|}} \,, 
\end{equation}
so that (\ref{41}) follows from (\ref{40}).
Assuming from now on $||\tau_3|-|\xi_3|| \ll |\xi_3|$ we have to prove
\begin{equation}
\label{42}
 \int_* m(\xi_1,\xi_2,\xi_3,\tau_1,\tau_2,\tau_3) \prod_{i=1}^3 \widehat{u}_i(\xi_i,\tau_i) d\xi d\tau \lesssim \prod_{i=1}^3 \|u_i\|_{L^2_{xt}} 
 \end{equation}
where
$$ m= \frac{(|\tau_2|+|\tau_3|) \chi_{||\tau_3|-|\xi_3|| \ll |\xi_3|}}{\langle \xi_1 \rangle^{1-l} \langle \tau_1 \rangle^{\frac{1}{2}-\epsilon} \langle \xi_2 \rangle^{s} \langle \tau_2 \rangle^{\frac{1}{2}+\epsilon} \langle \xi_3 \rangle^s \langle |\tau_3|-|\xi_3|\rangle^{\frac{1}{2}+\epsilon}} \, . $$
Since $\langle \tau_3 \rangle \sim \langle \xi_3 \rangle$ and $\tau_1+\tau_2+\tau_3=0$ we have 
\begin{equation}
\label{43}
|\tau_2| + |\tau_3| \lesssim \langle \tau_1 \rangle^{\frac{1}{2}-\epsilon} \langle \tau_2 \rangle^{\frac{1}{2}+\epsilon} +\langle \tau_1 \rangle^{\frac{1}{2}-\epsilon} \langle \xi_3 \rangle^{\frac{1}{2}+\epsilon} +\langle \tau_2 \rangle^{\frac{1}{2}+\epsilon} \langle \xi_3 \rangle^{\frac{1}{2}-\epsilon} \, .
\end{equation}
For the first term on the r.h.s. we have to show
$$
\big|\int\int uvw dx dt\big| \lesssim \|u\|_{X^{1-l,0}_{\tau=0}} \|v\|_{X^{s,0}_{\tau=0}} \|w\|_{X^{s,\frac{1}{2}+\epsilon}_{|\tau|=|\xi|}} \, ,
$$
which follows from Sobolev's multiplication law (\ref{SML}). For the other terms we use $l \ge 1$ so that $\langle \xi_1 \rangle^{l-1} \lesssim \langle \xi_2 \rangle^{l-1} + \langle \xi_3 \rangle^{l-1}$ and the second term on the r.h.s. reduces to the following estimates:
\begin{align}
\label{47}
\big|\int\int uvw dx dt\big| &\lesssim \|u\|_{X^{0,0}_{\tau=0}} \|v\|_{X^{s+1-l,\frac{1}{2}+\epsilon}_{\tau=0}} \|w\|_{X^{s-\frac{1}{2}-\epsilon,\frac{1}{2}+\epsilon}_{|\tau|=|\xi|}} \\
\label{48}
\big|\int\int uvw dx dt\big| &\lesssim \|u\|_{X^{0,0}_{\tau=0}} \|v\|_{X^{s,\frac{1}{2}+\epsilon}_{\tau=0}} \|w\|_{X^{s-l+\frac{1}{2}-\epsilon,\frac{1}{2}+\epsilon}_{|\tau|=|\xi|}} \, .
\end{align}
First we prove (\ref{47}). We obtain
\begin{align}
\label{47'}
\big|\int\int uvw dx dt\big| &\lesssim \|u\|_{L_x^2 L_t^2} \|v\|_{L_x^q L_t^{\infty}} \|w\|_{L_x^p L_t^2} \, .
\end{align}
Here $\frac{1}{p} = \frac{n-1}{2(n+1)} - \frac{m}{n}$ , where $m = s-\half-\frac{n-1}{2(n+1)} - \epsilon$ and $\frac{1}{q} = \half - \frac{1}{p}$ , so that by Sobolev $H_x^{m,\frac{2(n+1)}{n-1}} \hookrightarrow L_x^p$ . This implies by Prop. \ref{Lemma} 
$$ \|w\|_{L_x^p L_t^2} \lesssim \|w\|_{H^{m,\frac{2(n-1)}{n+1}}_x L^2_t} \lesssim \|w\|_{H^{m+\frac{n-1}{2(n+1)},\half+\epsilon}} = \|w\|_{H^{s-\half-\epsilon,\half+\epsilon}} \, . $$
Moreover an easy calculation shows that $\frac{1}{q} \ge \half - \frac{s+1-l}{n} \, \Leftrightarrow \, l \le 2s-\frac{n}{2}+1-\epsilon$ , which holds by assumption., so that by Sobolev $H^{s+1-l}_x \hookrightarrow L^q_x$ and thus 
$$ \|v\|_{L^q_x L^{\infty}_t} \lesssim \|v\|_{X^{s+1-l,\half+\epsilon}_{\tau =0}} \, . $$
This implies (\ref{47}).

Next we prove (\ref{48}). We apply (\ref{47'}) with the choice $\frac{1}{p} = \half - \frac{s}{n}$ and $\frac{1}{q}=\frac{s}{n}$ . This implies by Sobolev $H^{k,\frac{2(n+1)}{n-1}}_x \hookrightarrow L^q_x$ , if $k= \frac{n(n-1)}{2(n+1)} -s$ . One easily checks that $k+\frac{n-1}{2(n+1)} \le s-l+\half-\epsilon \, \Leftrightarrow \, l \le 2s-\frac{n}{2}+1-\epsilon$ , which we assumed. Consequently by Prop. \ref{Lemma} we obtain
$$ \|w\|_{L^q_x L^2_t} \lesssim  \|w\|_{H^{k,\frac{2(n+1)}{n-1}}_x L^2_t} \lesssim \|w\|_{H^{k+\frac{n-1}{2(n+1)},\half+\epsilon}} \lesssim \|w\|_{H^{s-l+\half,\half+\epsilon}} \ $$
and by Sobolev $\|v\|_{L^p_x L^{\infty}_t} \lesssim \|v\|_{X^{s,\half+\epsilon}_{\tau =0}}$ , so that (\ref{48}) is proven.

For the third term on the r.h.s. of (\ref{43}) we have to show
\begin{align*}
\big|\int\int uvw dx dt\big| &\lesssim \|u\|_{X^{0,\frac{1}{2}-\epsilon}_{\tau=0}} \|v\|_{X^{s+1-l,0}_{\tau=0}} \|w\|_{H^{s-\frac{1}{2}+\epsilon,\frac{1}{2}+\epsilon}} \\
\big|\int\int uvw dx dt\big| &\lesssim \|u\|_{X^{0,\frac{1}{2}-\epsilon}_{\tau=0}} \|v\|_{X^{s,0}_{\tau=0}} \|w\|_{H^{s-l+\frac{1}{2}+\epsilon,\frac{1}{2}+\epsilon}} \, .
\end{align*}
The first estimate follows from
$$\big|\int\int uvw dx dt\big| \lesssim \|u\|_{L_x^2 L_t^{\infty-}} \|v\|_{L_x^q L_t^2} \|w\|_{L_x^p L_t^{2+}}  $$
 exactly as for the proof of (\ref{47}) with (up to an $\epsilon$) the same parameters.
The second estimate follows similarly by choosing $p$ and $q$ (up to an $\epsilon$) as for the proof of (\ref{48}).

We now come to the proof of (\ref{27}) and remark that we may assume now that both functions $\phi_1$ and $\phi_2$ are supported in $||\tau|-|\xi|| \ll |\xi|$ , because otherwise (\ref{27}) is an immediate consequence of (\ref{41}) and (\ref{41'}). Thus (\ref{27}) follows if we can prove the following estimate:
$$ \int_* m(\xi_1,\xi_2,\xi_3,\tau_1,\tau_2,\tau_3) \prod_{i=1}^3 \widehat{u}_i(\xi_i,\tau_i) d\xi d\tau \lesssim \prod_{i=1}^3 \|u_i\|_{L^2_{xt}} \, , $$
where
$$m= \frac{|\tau_3|\chi_{||\tau_2|-|\xi_2|| \ll |\xi_2|} \chi_{||\tau_3|-|\xi_3|| \ll |\xi_3|}}{\langle \xi_1 \rangle^{1-l} \langle \tau_1 \rangle^{\frac{1}{2}-\epsilon} \langle \xi_2 \rangle^s \langle |\tau_2|-|\xi_2| \rangle^{\frac{1}{2}+\epsilon} \langle \xi_3 \rangle^s \langle |\tau_3|-|\xi_3|\rangle^{\frac{1}{2}+\epsilon}} \, . $$
Since $\langle \tau_3 \rangle \sim \langle \xi_3 \rangle$ , $\langle \tau_2 \rangle \sim \langle \xi_2 \rangle$ and $\tau_1+\tau_2+\tau_3=0$ we obtain
$$
|\tau_3| \lesssim \langle \tau_1 \rangle^{\frac{1}{2}-\epsilon} \langle \xi_3 \rangle^{\frac{1}{2}+\epsilon} +\langle \xi_2 \rangle^{\frac{1}{2}-\epsilon} \langle \xi_3 \rangle^{\frac{1}{2}+\epsilon} \, . $$
The first term is taken care of by the estimate
$$\big|\int \int uvw dx dt\big| \lesssim \|u\|_{X^{1-l,0}_{\tau=0}} \|v\|_{H^{s,\frac{1}{2}+\epsilon}} \|w\|_{H^{s-\frac{1}{2}-\epsilon,\frac{1}{2}+\epsilon}} \, ,$$
which is equivalent to
$$\|vw\|_{H^{l-1,0}} \lesssim \|v\|_{H^{s,\half+\epsilon}} \|w\|_{H^{s-\half-\epsilon,\half+\epsilon}} \, .$$
This is true by Prop. \ref{Prop.1.2} and Prop. \ref{Prop.2} with the parameters $s_0=1-l$ , $s_1=s$ , $s_2=s-\half-\epsilon$ , so that $s_0+s_1+s_2 = 2s-l+\half-\epsilon > \frac{n-1}{2} \ge\frac{n+1}{4}$ under our assumption $ l < 2s-\frac{n}{2}+1$ , and also $s_1+s_2 > \half$ . \\
In order to treat the second term on the right hand side we assume w.l.o.g. $|\xi_2| \ge |\xi_3|$, so that $\langle \xi_1 \rangle^{l-1} \lesssim \langle \xi_2 \rangle^{l-1}$ , so that it suffices to show:
$$\big|\int \int uvw dx dt\big| \lesssim \|u\|_{X^{0,\half-\epsilon}_{\tau=0}} \|v\|_{H^{s-l+\half,\frac{1}{2}+\epsilon}} \|w\|_{H^{s-\frac{1}{2},\frac{1}{2}+\epsilon}} \, ,$$
This is shown as follows:
$$\big|\int \int uvw dx dt\big| \lesssim \|u\|_{L^2_x L^{\infty-}_t} \|v\|_{L^{q_1}_x L^2_t} \|w\|_{L^{q_2}_x L^{2+}_t} \, .$$
We choose $q_1$ such that $\frac{1}{q_1}\ge \frac{n-1}{2(n+1)} - \frac{k_1}{n}$ with $k_1 + \frac{n-1}{2(n+1)} = s-l+\half$ , which is equivalent to $\frac{1}{q_1} \ge \frac{n-1}{2n}- \frac{s-l+\half}{n}$ . This implies $H^{k_1,\frac{2(n+1)}{n-1}} \hookrightarrow L^{q_1}$, so that by Prop. \ref{Lemma} we obtain
$$\|v\|_{L^{q_1}_x L^2_t} \lesssim \|v\|_{H^{k_1,\frac{2(n+1)}{n-1}}_x L^2_t} \lesssim \|v\|_{H^{s-l+\half,\half+\epsilon}} \, . $$
Moreover we want to choose  $q_2$ such that $\frac{1}{q_2}\ge \frac{n-1}{2(n+1)} - \frac{k_2}{n}$ with $k_2 + \frac{n-1}{2(n+1)} < s-\half$, which means that $\frac{1}{q_2} > \frac{n-1}{2n}- \frac{s-\half}{n}$ , thus as before
$$\|v\|_{L^{q_2}_x L^{2+}_t} \lesssim \|v\|_{H^{k_2,\frac{2(n+1)}{n-1}}_x L^{2+}_t} \lesssim \|v\|_{H^{s-\half,\half+\epsilon}} \, . $$
This choice of the parameters $q_1$ and $q_2$ is possible, if 
$\half = \frac{1}{q_1}+\frac{1}{q_2} > \frac{n-1}{n} - \frac{2s-l}{n}$ , which is equivalent to
our assumption $l < 2s-\frac{n}{2}+1$ . 

This completes the proof of (\ref{27}). \\[0.5em]
{\bf Proof of (\ref{29}):} This proof is similar to a related estimate for the Yang-Mills equation given by Tao \cite{T1}. We have to show
$$
\int_* m(\xi_1,\xi_2,\xi_3,\tau_1,\tau_2,\tau_3) \prod_{i=1}^3 \widehat{u}_i(\xi_i,\tau_i)  d\xi d\tau \lesssim \prod_{i=1}^3 \|u_i\|_{L^2_{xt}} \, , 
$$
where 
$$ m = \frac{(|\xi_2|+|\xi_3|) \langle \xi_1 \rangle^{s-1} }{\langle |\tau_1|-|\xi_1|) \rangle^{\frac{1}{2}-2\epsilon}\langle \xi_2 \rangle^s \langle |\tau_2| - |\xi_2|\rangle^{\frac{1}{2}+\epsilon}  \langle \xi_3 \rangle^l \langle \tau_3 \rangle^{\frac{1}{2}+\epsilon-}} \, .$$
Case 1: $|\xi_2| \lesssim |\xi_1|$ ($\Rightarrow$ $|\xi_2|+|\xi_3| \lesssim |\xi_1|$). \\
We ignore the factor $\langle |\tau_1| - |\xi_1| \rangle^{\frac{1}{2}-2\epsilon}$ and use the averaging principle (\cite{T}, Prop. 5.1) to replace $m$ by
$$ m' = \frac{ \langle \xi_1 \rangle^s \chi_{||\tau_2|-|\xi_2||\sim 1} \chi_{|\tau_3| \sim 1}}{ \langle \xi_2 \rangle^s \langle \xi_3 \rangle^l} \, . $$
Let now $\tau_2$ be restricted to the region $\tau_2 =T + O(1)$ for some integer $T$. Then $\tau_1$ is restricted to $\tau_1 = -T + O(1)$, because $\tau_1 + \tau_2 + \tau_3 =0$, and $\xi_2$ is restricted to $|\xi_2| = |T| + O(1)$. The $\tau_1$-regions are essentially disjoint for $T \in {\mathbb Z}$ and similarly the $\tau_2$-regions. Thus by Schur's test (\cite{T}, Lemma 3.11) we only have to show
\begin{align*}
 &\sup_{T \in {\mathbb Z}} \int_* \frac{\langle \xi_1 \rangle^s \chi_{\tau_1=-T+O(1)} \chi_{\tau_2=T+O(1)} \chi_{|\tau_3|\sim 1} \chi_{|\xi_2|=|T|+O(1)}}{\langle \xi_2 \rangle^s  \langle \xi_3 \rangle^l} \prod_{i=1} \widehat{u}_i(\xi_i,\tau_i)  d\xi d\tau  \\
 & \hspace{25em} \lesssim \prod_{i=1}^3 \|u_i\|_{L^2_{xt}} \, . 
\end{align*}
The $\tau$-behaviour of the integral is now trivial, thus we reduce to
\begin{equation}
\label{55}
\sup_{T \in {\mathbb N}} \int_{\sum_{i=1}^3 \xi_i =0}  \frac{ \langle \xi_1 \rangle^s \chi_{|\xi_2|=T+O(1)}}{ \langle T \rangle^s  \langle \xi_3 \rangle^{l}} \widehat{f}_1(\xi_1)\widehat{f}_2(\xi_2)\widehat{f}_3(\xi_3)d\xi \lesssim \prod_{i=1}^3 \|f_i\|_{L^2_x} \, .
\end{equation}
It only remains to consider the following two cases: \\
Case 1.1: $|\xi_1| \sim |\xi_3| \gtrsim T$. We have to show 
$$
\sup_{T \in {\mathbb N}} \int_{\sum_{i=1}^3 \xi_i =0}  
\frac{\chi_{|\xi_2|=T+O(1)}}{ T^l} \widehat{f}_1(\xi_1)\widehat{f}_2(\xi_2)\widehat{f}_3(\xi_3)d\xi \lesssim \prod_{i=1}^3 \|f_i\|_{L^2_x} \, . $$
The l.h.s. is bounded by
\begin{align*} 
& \sup_{T \in{\mathbb N}} \frac{1}{T^l} \|f_1\|_{L^2} \|f_3\|_{L^2} \| {\mathcal F}^{-1}(\chi_{|\xi|=T+O(1)} \widehat{f}_2)\|_{L^{\infty}({\mathbb R}^3)} \\
&\lesssim \sup_{T \in{\mathbb N}} \frac{1}{ T^l} 
\|f_1\|_{L^2} \|f_3\|_{L^2} \| \chi_{|\xi|=T+O(1)} \widehat{f}_2\|_{L^1({\mathbb R}^3)} \\
&\lesssim \hspace{-0.1em}\sup_{T \in {\mathbb N}} 
\frac{T^{\frac{n-1}{2}}}{T^l}  \prod_{i=1}^3 \|f_i\|_{L^2} \lesssim \hspace{-0.1em}
\prod_{i=1}^3 \|f_i\|_{L^2}
\end{align*}
for $l \ge \frac{n-1}{2}$ . \\
Case 1.2: $|\xi_1| \sim T \gtrsim |\xi_3|$. 
In this case it suffices to show
$$
\sup_{T \in {\mathbb N}} \int_{\sum_{i=1}^3 \xi_i =0}  \frac{\chi_{|\xi_2|=T+O(1)}}{ \langle \xi_3 \rangle^{l}} \widehat{f}_1(\xi_1)\widehat{f}_2(\xi_2)\widehat{f}_3(\xi_3)d\xi \lesssim \prod_{i=1}^3 \|f_i\|_{L^2_x} \, .
$$
 An elementary calculation shows  that the l.h.s. is bounded by
\begin{align*}
 \sup_{T \in{\mathbb N}} \| \chi_{|\xi|=T+O(1)} \ast \langle \xi \rangle^{-2l}\|^{\frac{1}{2}}_{L^{\infty}(\mathbb{R}^3)} \prod_{i=1}^3 \|f_i\|_{L^2_x} \lesssim \prod_{i=1}^3 \|f_i\|_{L^2_x} \, ,
\end{align*}
using that $l > \frac{n-1}{2}$ .

The proof of (\ref{29}) is complete. \\[0.5em]
{\bf Proof of (\ref{30}):}  We estimate by Sobolev's multiplication law (\ref{SML}), Prop. \ref{Prop.2} and Prop. \ref{Prop.1.2} , using $s> \frac{n}{2}-\frac{3}{4}$ :
\begin{align*}
&\| A \phi_1 \phi_2 \|_{H^{r-1,\frac{n}{2}-1-r+2\epsilon}} \lesssim \|A \phi_1 \phi_2\|_{L^2_t H^{r-1}_x} \lesssim \|A\|_{L^{\infty}_t H^r_x} \|\phi_1 \phi_2\|_{L^2_t H^{\frac{n}{2}-1+}_x} \\
& \lesssim \|A\|_{H^{r,\frac{1}{2}+\epsilon}} \|\phi_1\|_{H^{s,\frac{1}{2}+\epsilon}}
\|\phi_2\|_{H^{s,\frac{1}{2}+\epsilon}} \, .
\end{align*}
Similarly we also obtain for $r \le \frac{n-1}{2}$ , $l > \frac{n-1}{2}$ and $s > \frac{n}{2}-\frac{3}{4}$ :
\begin{align*}
&\|A \phi_1 \phi_2\|_{L^2_t H^{r-1}_x} 
\lesssim \|A\|_{L^{\infty}_t H^{\frac{n-1}{2}}_x}  \|\phi_1 \phi_2\|_{L^2_t H^{r-\half+}_x}\\
 & \lesssim \| A\|_{X^{l,\frac{1}{2}+\epsilon-}_{\tau=0}} \|\phi_1\|_{H^{s,\frac{1}{2}+\epsilon}} \|\phi_2\|_{H^{s,\frac{1}{2}+\epsilon}} \, .
\end{align*}
{\bf Proof of (\ref{30'}):}  By Sobolev's multiplication rule (\ref{SML}) and $ l \ge \frac{n-1}{2}$ we obtain 
\begin{align*}
\|A_1 A_2 \phi\|_{L^2_t H^{s-1}_x} &\lesssim \|A_1 A_2\|_{L^2_t H^{\frac{n}{2}-1+}_x} \|\phi\|_{L^{\infty}_t H^s_x} \\
& \lesssim \|A_1\|_{L^4_t H^{\frac{n}{2}-\half+}_x }
					\|A_2\|_{L^4_t H^{\frac{n}{2}-\half+}_x} 
					\|\phi\|_{L^{\infty}_t H^s_x} \\
& \lesssim \| A_1\|_{X^{l,\frac{1}{2}+\epsilon-}_{\tau=0}} \|  A_2\|_{X^{l,\frac{1}{2}+\epsilon-}_{\tau=0}} \|\phi\|_{H^{s,\frac{1}{2}+\epsilon}} \, .
\end{align*}
Next for $\frac{n}{2}-\half > r > \frac{n}{2}-1$ Prop. \ref{Prop.2} or Prop. \ref{Prop.1.2} imply :
\begin{align*}
\|A_1 A_2 \phi\|_{X^{s-1,-\frac{1}{2}+}_{|\tau|=|\xi|}} & \lesssim \|A_1 A_2\|_{H^{r-\half+,0}}  \|\phi\|_{H^{s,\frac{1}{2}+}}  \\
& \lesssim
\|  A_1\|_{H^{\frac{n}{2}-1+,\frac{1}{2}+}} \| A_2\|_{H^{\frac{n}{2}-1+,\frac{1}{2}+}} \|\phi\|_{H^{s,\frac{1}{2}+}}\\
&\lesssim 
\|  A_1\|_{H^{r,\frac{n}{2}-r}} \| A_2\|_{H^{r,\frac{n}{2}-r+}} \|\phi\|_{H^{s,\frac{1}{2}+}}\, .
\end{align*}
Finally we also obtain
\begin{align*}
\|A_1 A_2\|_{H^{r-\half+,0}}
&\lesssim \| A_1\|_{L^{\infty}_t H^r_x} \|A_2\|_{L^2_t H^l_x} \lesssim
\|  A_1\|_{H^{r,\frac{1}{2}+}} \| A_2\|_{X^{l,\frac{1}{2}+}_{\tau=0}} \\
&\lesssim 
\|  A_1\|_{H^{r,\frac{n}{2}-r}} \| A_2\|_{X^{l,\frac{1}{2}+}_{\tau=0}}\, ,
\end{align*}
by (\ref{SML}) under our assumptions $l > \frac{n-1}{2}$ and $r < \frac{n}{2}-\half$ .

This completes the proof of (\ref{30'}) and  part 1 of Proposition \ref{Theorem}. 
\end{proof}

Now we eliminate the assumption $a_0^{cf} = 0$ .
\begin{proof}[{\bf Proof of Theorem \ref{Theorem'}}]
We use Proposition \ref{Theorem} to construct a unique solution $(\phi',A')$ of the Cauchy problem for 
(\ref{1.11}),(\ref{1.12}),(\ref{1.13}) with initial conditions
$ \phi'(0)= e^{-i \chi} \phi_0$ , $(\partial_t \phi)(0) = e^{-i \chi} \phi_1$ , $A'(0) = a_0^{df}$ , $(\partial_t A)(0)= a_1$ , where $a_0 \in H^r$ , $a_1 \in H^{r-1}$ , $\phi_0 \in H^s$ , $\phi_1 \in H^{s-1}$ and the compatibility condition (\ref{CC}) is satisfied.
Here $\chi := -(-\Delta)^{-1}  div \,a_0$  is chosen such that $\nabla \chi = a^{cf}(0)$. The assumptions for the data in Prop. \ref{Theorem} are now shown to be satisfied. It is immediately clear that $A'^{cf}(0) = 0$ and also $A'(0) \in H^r$,  $(\partial_t A')(0) \in H^{r-1}$ . In order to show the regularity of the data for $\phi'$ we start with the estimate
$$ \|\nabla( uv)\|_{H^r} \le c_1 \|\nabla u \|_{H^r} \|\nabla v \|_{H^r} \, , $$
which holds for $r>\frac{n}{2}-1$ by (a variant of) (\ref{SML}). This implies
\begin{align}
\label{L}
\| \nabla (e^{i\chi}) \|_{H^r} & = \| \nabla(\sum_{k=0}^{\infty} \frac{(i \chi)^k}{k !})  \|_{H^r} \le  \sum_{k=0}^{\infty} \frac{1}{k!}\|\nabla(\chi^k)\|_{H^r} \\
\nonumber
&\le \sum_{k=0}^{\infty} \frac{c_1^{k-1} \|\nabla \chi\|_{H^r}^k}{k!} = c_1^{-1}\exp(c_1 \|\nabla \chi\|_{H^r})  < \infty \, .
\end{align}
Thus by (\ref{SML}) using $r>\frac{n}{2}-1$ :
$$ \|\phi'(0) \|_{H^s} = \|e^{i\chi} \phi_0\|_{H^s} \lesssim \|\nabla(e^{i\chi})\|_{H^r} \|\phi_0\|_{H^s} < \infty $$
and similarly also $(\partial_t \phi')(0) \in H^{s-1}$. The compatibility condition is also preserved, as one easily shows.

Consider now the gauge transformation
$$A'_{\mu} \to A_{\mu} = A'_{\mu} + \partial_{\mu} \chi \, , \, \phi' \to \phi = e^{i \chi} \phi' \, , \, D'_{\mu} \to D_{\mu} = \partial_{\mu} + i A_{\mu} \, .$$
It certainly preserves the temporal gauge, because $\chi$ is independent of the time.
This leads to a solution $(A,\phi)$ of (\ref{1.11}),(\ref{1.12}),(\ref{1.13}) with initial conditions
$A(0) = a_0^{df} + \nabla \chi = a_0^{df} + a_0^{cf} = a_0$ , $(\partial_t A)(0) = a_1$ , $\phi(0) = \phi_0$ , $(\partial_t \phi)(0) = \phi_1$ . What remains to be shown is that the regularity of the solution is preserved. It is easy to see that $A$ has the same regularity as $A'$.
Let now $\psi$ be a smooth function with $\psi(t) =1$ for $0\le t \le T$ and $\psi(t) = 0$ for $t \ge 2T$. By Lemma \ref{Lemma4.3} below and (\ref{L}) we obtain for $s>\frac{n}{2}-\frac{3}{4}$ and $r > \frac{n}{2}-1$ :
\begin{align*}
\|e^{i\chi} \phi'_{\pm}\|_{X^{s,\frac{1}{2}+\epsilon}_{\pm} [0,T]} &
\lesssim \|\nabla(e^{i\chi}) \psi\|_{X^{r,\frac{1}{2}+\epsilon}_{\pm}} 
\|\phi'_{\pm}\|_{X^{s,\frac{1}{2}+\epsilon}_{\pm}[0,T]} \\
&\lesssim \|\nabla(e^{i\chi})\|_{H^r} \|\phi'_{\pm}\|_{X^{s,\frac{1}{2}+\epsilon}_{\pm}[0,T]} \\
& \lesssim c_1^{-1}\exp(c_1 \|a_0^{cf}\|_{H^r})  \|\phi'_{\pm}\|_{X^{s,\frac{1}{2}+\epsilon}_{\pm}[0,T]} < \infty \, ,
\end{align*}
so that the regularity of $\phi$ is also preserved.
\end{proof}
In the last proof we used the following lemma.
\begin{lemma}
\label{Lemma4.3}
The following estimate holds for $r+1 \ge s>\frac{n}{2}-\frac{3}{4}$ , $r > \frac{n}{2}-1$ and $\epsilon >0$ sufficiently small:
$$ \|uv\|_{X^{s,\frac{1}{2}+\epsilon}_{\pm}} \lesssim \|\nabla u\|_{X^{r,\frac{1}{2}+\epsilon}_{\pm}} \|v\|_{X^{s,\frac{1}{2}+\epsilon}_{\pm}} \, . $$
\end{lemma}
\begin{proof}
By Tao \cite{T}, Cor. 8.2 we may replace $\nabla$ by $\Lambda$ so that it suffices to prove
$$  \|uv\|_{X^{s,\frac{1}{2}+\epsilon}_{\pm}} \lesssim \| u\|_{X^{r+1,\frac{1}{2}+\epsilon}_{\pm}} \|v\|_{X^{s,\frac{1}{2}+\epsilon}_{\pm}} \, . $$
We start with the elementary estimate
$$|(\tau_1 + \tau_2)\mp |\xi_1+\xi_2|| \le |\tau_1 \mp |\xi_1|| + |\tau_2 \mp |\xi_2|| + |\xi_1| + |\xi_2| - |\xi_1+\xi_2| \, . $$
Assume now w.l.o.g. $|\xi_2|\ge|\xi_1|$. We have 
$$|\xi_1|+|\xi_2|-|\xi_1+\xi_2| \le |\xi_1|+|\xi_2| + |\xi_1| - |\xi_2| =  2|\xi_1| \, ,$$
so that $$|(\tau_1 + \tau_2)\mp |\xi_1+\xi_2||  \le |\tau_1 \mp\xi_1| + |\tau_2 \mp |\xi_2|| + 2\min(|\xi_1|,|\xi_2|) \, . $$
Using Fourier transforms by standard arguments it thus suffices to show the following three estimates:
\begin{align*}
\|uv\|_{X_{\pm}^{s,0}} & \lesssim \|u\|_{X^{r+1,0}_{\pm}} \|v\|_{X^{s,\frac{1}{2}+\epsilon}_{\pm}} \\
\|uv\|_{X_{\pm}^{s,0}} & \lesssim \|u\|_{X^{r+1,\frac{1}{2}+\epsilon}_{\pm}} \|v\|_{X^{s,0}_{\pm}} \\
\|uv\|_{X_{\pm}^{s,0}} & \lesssim \|u\|_{X^{r+\frac{1}{2}-\epsilon,\frac{1}{2}+\epsilon}_{\pm}} \|v\|_{X^{s,\frac{1}{2}+\epsilon}_{\pm}} 
\end{align*}
The first and second estimate easily follow from SML (\ref{SML}), whereas the last one is implied by Prop. \ref{Prop.2} and Prop. \ref{Prop.1.2} with the parameters $s_0=-s$ , $s_1= r+\half-\epsilon$ , $s_2=s$ , so that $s_0+s_1+s_2 > \frac{n-1}{2}$ and $s_1+s_2 = s+r+\half-\epsilon > n - \frac{5}{4} > \half$ . 
\end{proof}


\begin{thebibliography}{999999}
\bibitem[1]{AFS} P. d'Ancona, D. Foschi, and S. Selberg: {\sl Product estimates for wave-Sobolev spaces in 2+1 and 1+1 dimensions}. Contemporary Math. 526 (2010), 125-150
\bibitem[2]{AFS1} P. d'Ancona, D. Foschi, and S. Selberg: {\sl Atlas of products for wave-Sobolev spaces on ${\mathbb R}^{1+3}$ }. Transact. AMS 364 (2012), 31-63
\bibitem[3]{B} M. Beals: {\sl Self-spreading of singularities for solutions to semilinear wave equations}. Annals Math. 118 (1983), 187-214
\bibitem[4]{CP} M. Czubak and N. Pikula: {\sl Low regularity well-posedness for the 2D Maxwell-Klein-Gordon equation in the Coulomb gauge}. Comm. Pure Appl. Anal. 13 (2014), 1669-1683
\bibitem[5]{GV} J. Ginibre and G. Velo: {\sl Generalized Strichartz inequalities for the wave equation.} J. Functional Anal. 133 (1995), 60-68
\bibitem[6]{KRT} M. Keel, T. Roy and T. Tao: {\sl Global well-posedness of the Maxwell-Klein-Gordon equation below the energy norm}. Discrete Cont. Dyn. Syst. 30 (2011), 573-621
\bibitem[7]{KMBT} S. Klainerman and M. Machedon (Appendices by J. Bougain and  D. Tataru): {\sl Remark on Strichartz-type inequalities}. Int. Math. Res. Notices 1996, no. 5, 201-220
\bibitem[8]{KM} S. Klainerman and M. Machedon: {\sl On the Maxwell-Klein-Gordon equation with finite energy}. Duke Math. J. 74 (1994), 19-44
\bibitem[9]{KM2} S. Klainerman and M. Machedon: {\sl Estimates for null forms and the spaces $H_{s,\delta}$.} Int. Math. Res. Notices 1996, no. 17, 853-866
\bibitem[10]{MS} M. Machedon and J. Sterbenz: {\sl Almost optimal local well-posedness for the (3+1)-dimensional Maxwell-Klein-Gordon equations}. J. AMS 17 (2004), 297-359
\bibitem[11]{M} V. Moncrief: {\sl Global existence of Maxwell-Klein-Gordon fields in (2+1)-dimensional spacetime}. J. Math. Phys. 21 (1980), 2291-2296
\bibitem[12]{P} H. Pecher: {\sl Low regularity local well-posedness for the Maxwell-Klein-Gordon equations in Lorenz gauge}. Adv. Diff. Equ. 19 (2014), 359-386
\bibitem[13]{P1} H. Pecher: {\sl Unconditional global well-posedness in energy space for the Maxwell-Klein-Gordon system in temporal gauge}. Adv. Diff. Equ. 20 (2015), 1009-1032.
\bibitem[14]{P3} H. Pecher: {\sl Low regularity solutions for the (2+1)-dimensional Maxwell-Klein-Gordon equations in temporal gauge}. Comm. Pure Appl. Analysis 15 (2016), 2203-2219
\bibitem[15]{P4} H. Pecher: {\sl Low regularity local well-posedness for the Yang-Mills equation in Lorenz gauge}. arXiv:1703.01949
\bibitem[16]{S1} S. Selberg: {\sl Multilinear space-time estimates and applications to local exisatence theory for nonlinear wave equations}. PhD. Thesis Princeton, 1999
\bibitem[17]{S} S. Selberg: {\sl Anisotropic bilinear $L^2$ estimates related to the 3D wave equation}. Int. Math. Res. Not. 2008, art. ID rnn107
\bibitem[18]{ST} S. Selberg and A. Tesfahun: {\sl Finite energy global well-posedness of the Maxwell-Klein-Gordon system in Lorenz gauge}. Comm. PDE 35 (2010), 1029-1057
\bibitem[19]{T} T. Tao: {\sl Multilinear weighted convolutions of $L^2$-functions and applications to non-linear dispersive equations}. Amer. J. Math. 123 (2001), 838-908
\bibitem[20]{T1} T. Tao: {\sl Local well-posedness of the Yang-Mills equation in the temporal gauge below the energy norm}. J. Diff. Equ. 189 (2003), 366-382
\bibitem[21]{Y} J. Yuan: {\sl Global solutions of two coupled Maxwell systems in the temporal gauge}. Discr.  Cont. Dyn. Syst. 36 (2016), 1709-1719
\end{thebibliography}
\end{document}